\documentclass[12pt,a4paper]{article}

\usepackage{header}
\usepackage{abbrev}

\begin{document}

\title{The leapfrog algorithm as nonlinear Gauss--Seidel}

\author{Marco Sutti\,\orcidlink{0000-0002-8410-1372}\thanks{Mathematics Division, National Center for Theoretical Sciences, Taipei, Taiwan (\email{msutti@ncts.tw}).}
~and Bart Vandereycken\thanks{Section of Mathematics, University of Geneva, Geneva, Switzerland (\email{bart.vandereycken@unige.ch}).}}

\date{}

\maketitle

\begin{abstract}
Several applications in optimization, image, and signal processing deal with data that belong to the Stiefel manifold $ \Stnp $, that is, the set of $n\times p$ matrices with orthonormal columns. Some applications, like the Riemannian center of mass, require evaluating the Riemannian distance between two arbitrary points on $ \Stnp $. This can be done by explicitly constructing the geodesic connecting these two points.
An existing method for finding geodesics is the leapfrog algorithm of J. L. Noakes. This algorithm is related to the Gauss--Seidel method, a classical iterative method for solving a linear system of equations that can be extended to nonlinear systems. We propose a convergence proof of leapfrog as a nonlinear Gauss--Seidel method. Our discussion is limited to the case of the Stiefel manifold, however, it may be generalized to other embedded submanifolds.
We discuss other aspects of leapfrog and present some numerical experiments.



\end{abstract}

\section{Introduction} \label{sec:geodesic_exp_log}
The object of study in this paper is the compact Stiefel manifold, i.e., the set of orthonormal $n$-by-$p$ matrices
\[
   \Stnp = \lbrace X \in \R^{n \times p}: \ X\tr\! X = I_p \rbrace.
\]
Here, we are concerned with computing the Riemannian distance between two points on the Stiefel manifold. 
The distance between two points on a manifold is related to the concept of minimizing geodesic.
A geodesic $ \gamma \colon [0,t] \to \cM $ is a curve with zero acceleration, which generalizes the notion of straight lines in Euclidean space to a Riemannian manifold~\cite{AMS:2008}.
Geodesics allow us to introduce the \emph{Riemannian exponential} $ \mathrm{Exp}_{x}\colon T_{x}\cM \to \cM $ that maps a tangent vector $\xi = \dot{\gamma}(0) \in T_{x}\cM $ to the geodesic endpoint $\gamma(1) = y $ such that $ \mathrm{Exp}_{x}(\xi) = y$. 
The Riemannian exponential is a local diffeomorphism, i.e., it is locally invertible and its inverse is called the \emph{Riemannian logarithm} of $y$ at $x$ satisfying $ \mathrm{Log}_{x}(y) = \xi $.
The \emph{injectivity radius} at a point $x$ of a Riemannian manifold $\cM$ is the largest radius for which the exponential map $\mathrm{Exp}_{x}$ is a diffeomorphism from the tangent space to the manifold. The global injectivity radius of a manifold is the infimum of all the injectivity radii over all points of the manifold.
Given two points $ x $ and $ y $ on a manifold $ \cM $, if the Riemannian distance $ d(x,y) $ is smaller than $ \mathrm{inj}(\cM) $, then there exists a unique length-minimizing geodesic from $ x $ to $ y $.
For the Stiefel manifold, the injectivity radius is at least $ 0.89\,\pi $; see \protect{\cite[Eq.~(5.13)]{Rentmeesters:2013}}.

The distance on the Stiefel manifold is involved in numerous fields of applications, among which, computer vision \protect{\cite{Cetingul:2009,Sundaramoorthi:2011,Turaga:2011,Yin:2015}}, statistics \protect{\cite{Srivastava:2016}}, reduced-order models \protect{\cite{Amsallem:2011}}.
Given two points on the Stiefel manifold, our goal is to compute the length of the minimizing geodesic connecting them. For some manifolds, there are explicit formulas available for computing the distance.
For the Stiefel manifold, there is no closed-form solution known. In general, the problem of finding the distance given two points on a Riemannian manifold is related to the Riemannian logarithm. The problem of computing the Riemannian logarithm on the Stiefel manifold has already been tackled by several authors, who proposed some numerical algorithms. Rentmeesters \protect{\cite{Rentmeesters:2013}} and Zimmermann \protect{\cite{Zimmermann:2017,Zimmermann:2019}} proposed a similar algorithm which is only locally convergent and depends upon the definition of the (standard) matrix logarithm function. Bryner \protect{\cite{Bryner:2017}} and Sutti \protect{\cite{Sutti:2020}} considered standard shooting methods for boundary value problems. Zimmermann and H\"uper \protect{\cite{Zimmermann:2022}} recently reviewed the available methods for computing the Riemannian distance on the Stiefel manifold.

Another method for finding geodesics is the leapfrog algorithm introduced by Noakes \cite{Noakes:1998}; see also~\cite{Noakes:2022}. This method has global convergence properties, but it slows down when the solution is approached \protect{\cite[p.~2796]{Kaya:2008}}. 
Moreover, Noakes realized that his leapfrog algorithm was related to the Gauss--Seidel method \protect{\cite[p.~39]{Noakes:1998}}.
The link between leapfrog and nonlinear Gauss--Seidel was not further investigated, since there is no trace of this idea being developed in the other related papers \cite{Kaya:1997,Kaya:2008}.
In this paper, we will prove convergence of leapfrog as a nonlinear block Gauss--Seidel method. Even though our focus will be on $\Stnp$, most of our discussion may be generalized to other embedded submanifolds. 

A Riemannian metric has to be specified in order to turn $\Stnp$ into a Riemannian manifold, and in general different choices are possible.
In this paper, we consider the non-Euclidean \emph{canonical metric} inherited by $\Stnp$ from its definition as a quotient space of the orthogonal group \protect{\cite[Eq.~(2.39)]{Edelman:1998}}. Given $ Y \in \Stnp $ and $ \xi \in T_{Y}\Stnp $, the canonical metric reads
\begin{equation}\label{eq:formula_canonical_metric}
    g_{c}(\xi,\xi) = \trace\!\big( \xi\tr ( I - \tfrac{1}{2}YY\tr) \, \xi \big).
\end{equation}
Tangent vectors to the Stiefel manifold may be expressed in the form
\[
   \xi = Y_{0}\Omega + Y_{0\perp}K, \quad \text{with} \quad \Omega \in \mathcal{S}_{\mathrm{skew}}(p), \quad K\in\R^{(n-p)\times p},
\]
with $ \mathcal{S}_{\mathrm{skew}}(p) $ the vector space of $p$-by-$p$ skew-symmetric matrices.

An explicit formula for a geodesic with initial acceleration the tangent vector $ \xi $ and base point $ Y_{0} $ is \protect{\cite[Eq.~(2.42)]{Edelman:1998}}
\begin{equation}\label{eq:closed-form-sol-geodesic}
    Y(t) = Q \exp\!\left( \begin{bmatrix}
        \Omega   &   -K\tr \\
        K        &    O_{n-p}
    \end{bmatrix} t \right)
    \begin{bmatrix}
        I_{p} \\
        O_{(n-p)\times p}
    \end{bmatrix},
\end{equation}
with $ Q = \big[ Y_{0} \ Y_{0\perp} \big] $ and $Y_{0\perp}$ being any matrix whose range is $ (\mathrm{span}(Y_{0}))^{\perp}$.

Given two points $Y_0$, $Y_1$ on $\Stnp$ that are sufficiently close to each other, finding the distance between them is equivalent to finding the tangent vector $\xi^{\ast} \in T_{Y_0}\Stnp$ with the shortest possible length such that $ \Exp_{Y_0}(\xi^{\ast}) = Y_1 $ \protect{\cite{Lee:2018,boumal2020intromanifolds}}. 
The solution to this problem is equivalent to the Riemannian logarithm of $Y_{1}$ with base point $Y_{0}$, i.e., $ \xi^{\ast} = \mathrm{Log}_{Y_0}(Y_1) $. Given the endpoints $ Y_0 $ and $ Y_1 $, we do not know what the matrices $ \Omega $ and $ K $ in \eqref{eq:closed-form-sol-geodesic} are. So the problem becomes: Find the matrices $ \Omega $ and $ K $ such that the explicit formula \eqref{eq:closed-form-sol-geodesic} gives the endpoint $ Y_{1} $.

\section{Leapfrog algorithm}
When $X, Y \in \cM$ are sufficiently close, their connecting geodesic can be found by applying Newton's method to~\eqref{eq:closed-form-sol-geodesic} such that $Y(1)=Y$ with $Y(0) = X$. This is more generally known as single shooting\footnote{In this context, there is no need to solve an ordinary differential equation as in a normal shooting method, because we have the solution \eqref{eq:closed-form-sol-geodesic}. Hence, it is actually Newton's method, but we keep the shooting terminology because it is typical for boundary value problems.}. However, when $X$ and $Y$ are far apart, it is well known that single shooting will have difficulty finding the connecting geodesic. The main idea behind the leapfrog algorithm of Noakes~\cite{Noakes:1998} is to exploit the success of single shooting by subdividing the global problem into several local problems, where intermediate points $X_i \in \cM$ are introduced between $X$ and $Y$, for which the endpoint geodesic problem can be solved again by single shooting.
The algorithm then iteratively updates a piecewise geodesic to obtain a globally smooth geodesic between $X$ and $Y$. This idea is actually not new and goes back as early as 1963 by Milnor~\cite[III.\S16]{Milnor:1963}. It also resembles the better-known multiple shooting method for boundary value problems but it is different. The single and multiple shooting methods for the endpoint geodesic problem on the Stiefel manifold are discussed in \protect{\cite{Bryner:2017,Sutti:2020}}.

\subsection{Formal description of the algorithm}
In this section, we describe the leapfrog algorithm by following the presentation in~\cite{Noakes:1998}. Let $\cM$ be a $C^{\infty}$ path-connected Riemannian manifold. Consider a \emph{piecewise} (or \emph{broken}) \emph{geodesic} $ \omega_{X} $ joining $ X_{0} $ to $ X_{m-1} $, having $ m-1 $ geodesic segments. Assuming $ X_{i} $ and $ X_{i+1} $ are sufficiently close to each other, $ \omega_{X} $ is uniquely identified by the $ m $-tuple $ X = ( X_{0}, X_{1}, \ldots, X_{m-1} ) \in \cM^{m} $, where $ X_{i} $ are the junctions of the geodesic segments. The leapfrog algorithm now proceeds as follows: for $ i=1,\ldots, m-2 $, each $ X_i $ is mapped onto the minimizing geodesic joining $ X_{i-1} $ and $ X_{i+1} $. This achieves the largest possible decrease in length while keeping other variables fixed. Though there are several choices to do this, leapfrog maps $ X_i $ onto the midpoint of the geodesic joining $ X_{i-1} $ and $ X_{i+1} $. By iterating this procedure, the algorithm generates a sequence $\Omega = \left\lbrace \omega_{X^{(k)}}\colon [0,1] \to \cM \colon k = 0,1, \ldots \right\rbrace $ of broken geodesics whose lengths are decreasing. Figure \ref{fig:leapfrog} illustrates one iteration of the leapfrog algorithm. It is clear that leapfrog generates a sequence of broken geodesics $\omega_{X^{(k)}}$ that are defined from $ X^{(k)}$. In addition, the length of $\omega_{X^{(k)}}$ is non-increasing in $k$ since at each step two neighboring geodesics get replaced by one global geodesic connecting their endpoints.

\begin{figure}
\centering
\includegraphics[width=0.85\columnwidth]{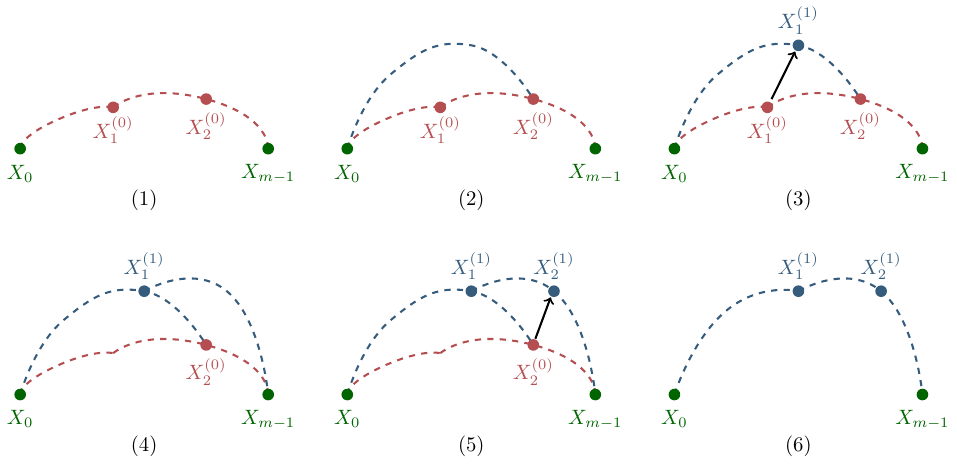}
\caption{Illustration of one full iteration of the leapfrog scheme for some non-Euclidean metric (the lengths for the Euclidean metric clearly increase during iteration).}
\label{fig:leapfrog}
\end{figure}

\subsection{Known results}
Let $ \cY $ be the set of all tuples $X=(X_{0},X_{1},\ldots,X_{m-1}) \in \cM^{m} $ satisfying $d(X_{i-1},X_{i})\leq \delta$ for all $i=1,2,\ldots,m-2$. In \cite[\S2]{Noakes:1998}, $\delta$ is related to the notion of Lebesgue number of an open cover. Here, we can assume that $ \delta $ is equal to $ \tfrac{1}{2} \inj\!\big(\cM\big) $, where $ \inj $ is the injectivity radius (see Section~\ref{sec:geodesic_exp_log}). Let $ \cF \colon \cY \to \cY $ represent one full leapfrog iteration and let $X^{\ast}$ be the limit of any convergent subsequence of $ S = \lbrace \cF^{k}(X^{(0)}) \colon k \geq 1 \rbrace $ with $X^{(0)} \in \cY$. By compactness, \cite{Noakes:1998} shows that at least one convergent subsequence of $S$ exists and that the limit of this subsequence are points that lie on a global geodesic connecting the endpoints $ X_0 $ and $ X_{m-1} $.
The following result is stated in \cite[Theorem~5.2]{Kaya:2008}.

\begin{theorem}\label{thm:kaya_noakes_convg_leapfrog}
    $S$ has a unique accumulation point.
\end{theorem}

The theorem guarantees convergence of the iterates $X^{(k)} = \cF(X^{(k-1)})$ with $X^{(0)} \in \cY$. From \cite[Lemma~3.2]{Noakes:1998} we also know that leapfrog will converge to a uniformly distributed $ m $-tuple $ X^{\ast} = ( X_{0}, X_{1}^{\ast} \ldots, X_{m-2}^{\ast}, X_{m-1} ) $, i.e., $ d(X_{i}^{\ast},X_{i+1}^{\ast}) $ are all equal, for $ i = 0, \ldots, m-2 $. In other words, at convergence, the geodesic segments connecting the junction points will all have the same length. 

An apparent drawback in the current theory is that it lacks a classical convergence proof as a fixed point iteration method, although leapfrog can be easily recognized as such. In the next section, we will provide the details of how to analyze leapfrog as a nonlinear block Gauss--Seidel method.

\section{Convergence of leapfrog as nonlinear Gauss--Seidel}\label{sec:lf_as_ngs}
Let $\cM = \Stnp$ with the Riemannian distance function $d$. The starting point is to realize that leapfrog solves the optimization problem
\[
    \min_{ X_{1}, \ldots, X_{m-2} \in \Stnp } F(X_{1}, \ldots, X_{m-2} ) \quad \text{with} \quad 
    F(X_{1}, \ldots, X_{m-2}) = \sum_{i=1}^{m-1} d^{2}(X_{i-1},X_{i}),
\]
by cyclically minimizing over each variable $X_{i}$ for $i=1,2,\ldots, m-2 $. Specifically, at the $k$th iteration, leapfrog updates $X_i^{(k-1)}$ by the minimizer of the problem
\begin{equation}\label{eq:constrained_opt_pb}
\begin{aligned}
   &\min_{X_{i} \in \Stnp } F( X_{1}^{(k)}, \ldots, X_{i-1}^{(k)}, X_{i}, X_{i+1}^{(k-1)}, \ldots, X_{m-2}^{(k-1)} ) \\ 
   &= \min_{X_{i} \in \Stnp } d^2(X_{i-1}^{(k)},X_{i}) + d^2(X_{i},X_{i+1}^{(k-1)}) + \text{constant}.
\end{aligned}
\end{equation}
Since $d$ is the Riemannian distance, this problem coincides with the definition of the Riemannian center of mass\footnote{The Riemannian center of mass was constructed in \protect{\cite{Grove:1973}}. As H. Karcher points out in \protect{\cite{Karcher:2014}}, ``Probably in 1990 someone renamed it without justification into \emph{karcher mean} and references to the older papers were omitted by those using the new name. (...) I think it is fair to say that a substantial amount of damage was caused by the renaming''. For this reason, in this paper, we decided to stick to the original name.}
between the two points $X_{i-1}^{(k)}$ and $X_{i+1}^{(k-1)}$; see~\cite[Eq.~(1.1)]{Karcher:1977}. For the compact Stiefel manifold, a Riemannian center of mass always exists, but it does not need to be unique \cite[p.~37]{Rentmeesters:2013}. However, a sufficient condition for uniqueness is $d(X_{i-1}^{(k)},X_{i+1}^{(k-1)}) < \inj\!\big(\Stnp\big)$, where $ \inj $ is the injectivity radius (see Section~\ref{sec:geodesic_exp_log}). This is true if all $X_i$ are close enough (we will make this more precise later). In that case, the unique solution that solves~\eqref{eq:constrained_opt_pb} is the midpoint of the minimizing geodesic between $X_{i-1}^{(k)}$ and $X_{i+1}^{(k-1)}$.
Leapfrog now proceeds to update the $X_i$ in a Gauss--Seidel fashion where the most recent $X_{i-1}^{(k)}$ is used to update $X_{i}^{(k-1)}$. This kind of optimization scheme is known as \emph{block coordinate descent method} of Gauss--Seidel type~\cite{Ortega:2000}.

\subsection{Nonlinear block Gauss--Seidel method}
Let us first consider the case of Gauss--Seidel in $ \R^{n} $.
Let the variable $ x \in \R^{n} $  be partitioned as $ x = (x_{1}, x_{2}, \ldots, x_{m}) $, where $ x_{i} \in \R^{q_{i}} $ and $ \sum_{i} q_{i} = n $, and group correspondingly the components of $  \widetilde{F} \colon D \subset \R^{n} \to \R^{n} $ into mappings $ \widetilde{F}_{i} \colon \R^{n} \to \R^{q_{i}} $, $ i = 1, \ldots, m $. 
The minimizers of the function $ \widetilde{F}(x) $ satisfy the first-order optimality condition $ \nabla \widetilde{F}(x) = 0 $. Let us define $ \cG_{i} = \nabla \widetilde{F}_{i} $, $ i = 1, \ldots, m $.
If we interpret the linear Gauss--Seidel iteration in terms of obtaining $ x_{i}^{(k)} $ as the solution of the $i$th equation of the system with the other $ m-1 $ block variables held fixed, then we may immediately consider the same prescription for nonlinear equations \cite[p.~219]{Ortega:2000}. Then solving
\begin{equation}\label{eq:ith_nonlinear_eq}
   \cG_{i}( x_{1}^{(k)},\ \ldots,\ x_{i-1}^{(k)}, \ y,\ x_{i+1}^{(k-1)}, \ \ldots,\  x_{m}^{(k-1)}) = 0
\end{equation}
for $ y $ and defining $ x_{i}^{(k)} = y $ describes a nonlinear block Gauss--Seidel process in which a complete iteration requires the solution of $ m $ nonlinear systems of dimensions $ q_{i} $, $ i = 1, \ldots, m $; see~\cite[p.~225]{Ortega:2000}.
The convergence theory in \protect{\cite{Ortega:2000}} applies only to functions whose domain of definition is Euclidean space $\R^{n}$. This theory cannot be applied to functions that are defined on manifolds. For instance, the Riemannian distance $ d $ is only defined on a subset of $\R^{n}$, i.e., the embedded submanifold. For this reason, in the next section, we will introduce a smooth extension of the Riemannian distance function that can also be evaluated for points that do not belong to the manifold.

\subsection{Extended objective function} \label{sec:ext_dist_function} 
As we have seen above, leapfrog solves in an alternating way the problem
\begin{equation*}
    \min_{X_1,\ldots, X_{m-2} \in \Stnp} F(X_{1}, \ldots, X_{m-2} ) = \sum_{i=1}^{m-1} d^2(X_{i-1},X_{i}),
\end{equation*}
where $X_0$ and $X_{m-1}$ are the fixed endpoints. This objective function $F$ is only defined on the manifold $\Stnp$.
In this section, we will identify an \emph{extended objective function} $\widetilde F$ that is defined on $\Rnp$ for which the standard nonlinear block Gauss--Seidel method produces the same iterates as the leapfrog algorithm. The key result of this section is stated in Prop.~\ref{prop:leapfrog_is_GS}. This will allow us to analyze the convergence of leapfrog using standard results for nonlinear Gauss--Seidel.

We claim the extended cost function can be chosen as
\begin{equation*}
    \min_{X_1,\ldots, X_{m-2} \in \Rnp} \widetilde F(X_{1}, \ldots, X_{m-2} ) = \sum_{i=1}^{m-1} \widetilde  d^2(X_{i-1},X_{i}),
\end{equation*}
with \emph{extended distance function}
\begin{equation}\label{eq:ext_dist_function}
\widetilde  d^2(\widetilde  X,   \widetilde  Y) = \begin{cases}
  d^{2}(\PSt \widetilde  X,\PSt \widetilde  Y)  + \| \widetilde X - \PSt \widetilde X \|_{\F}^{2} + \| \widetilde Y - \PSt \widetilde Y\|_{\F}^{2}, \\ \quad\quad\quad\qquad\qquad\qquad\qquad \text{if $\sigma_p(\widetilde X) > 0$ and $\sigma_p(\widetilde Y) > 0$;}\\
 +\infty, \qquad\qquad\qquad\quad\quad\quad \ \  \text{otherwise,}
 \end{cases}
\end{equation}
where $ \PSt $ denotes the orthogonal projector onto the Stiefel manifold, and $ \sigma_{p} $ is the smallest singular value.

The condition $\sigma_p(\widetilde  X)>0$ is equivalent to the existence of a unique best approximation of $\widetilde  X$ in $\Stnp$. In other words, $\PSt \widetilde  X$ is well defined. 
Concretely, we can define the projector $ \PSt \colon \Rnp \to \Stnp $ by $ \PSt(Z) = Z(Z\tr Z)^{-1/2} $, that is, the orthogonal factor of the polar decomposition of $ Z $ \cite[p.~58]{AMS:2008}. Figure \ref{fig:extended_distance} illustrates the extended distance function $ \widetilde  d^2(\widetilde  X,   \widetilde  Y)  $.

\begin{figure}
   \centering
   \includegraphics[width=0.55\columnwidth]{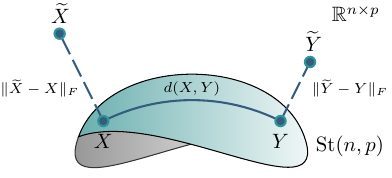}
   \caption{The extended distance function.} \label{fig:extended_distance}
\end{figure}

\subsection{Leapfrog as nonlinear Gauss--Seidel}
In order to show that nonlinear Gauss--Seidel applied to $\widetilde F$ is equivalent to leapfrog for $F$, we need a few lemmas. The first one is a known result that addresses the problem of how close the points on $\Stnp$ need to be so that their connecting geodesic is unique.

\begin{lemma}\label{lemma:geodesic_exists_for_close_neighbours}
 Let $X,Y \in \Stnp$ such that $d(X,Y)\leq \delta_g $, with $ \delta_g = 0.89 \, \pi $. Then there exists a unique minimizing geodesic between $X$ and $Y$. As a consequence, also the Riemannian center of mass between $X$ and $Y$ exists and is uniquely defined.
\end{lemma}

\begin{remark}\label{rmk:distance} We can compare the Riemannian and Euclidean distances between $X$ and $ Y \in \Stnp$ asymptotically in the following way\footnote{For the Riemannian distance $d_\mathrm{e}$ based on the embedded metric, it is easy to see that $ \| X - Y \|_{\F} \leq d_\mathrm{e}(X,Y) $ since the Euclidean length of a geodesic on $\Stnp$ is always larger than that of a straight line.}. 
From the expansion of the canonical distance in~\eqref{eq:expansion_canonical_dist}, it is clear that
\[
   d(X,Y) \leq \| X - Y \|_{\F} + \cO (\| X - Y \|^{2}_{\F}) \quad \text{for} \quad \| X - Y \|_{\F} \to 0.
\]
By neglecting $ \cO (\| X - Y \|^{2}_{\F}) $, we thus have $ d(X,Y) \lesssim \| X - Y \|_{\F} $. In particular, $\|X - Y\|_{\F} \leq \delta_g $ implies $ d(X,Y) \lesssim \delta_g$.
\end{remark}

Let $X_{i-1}, X_{i+1} \in \Stnp$. Denote 
\[ 
F_i(Y) = d^2(X_{i-1},Y) + d^2(Y,X_{i+1}), \qquad \widetilde F_i(\widetilde Y) = \widetilde d^2(X_{i-1}, \widetilde Y) + \widetilde d^2(\widetilde Y,X_{i+1}),
\]
where $X_{i-1}$, $ X_{i+1}$ are constant and hidden in the notation.

\begin{lemma}\label{lemma:substep_leapfrog_equals_substep_GS}
With the notation from above assume that $d(X_{i-1},X_{i+1})\leq \delta_g$, then the $i$th substep of leapfrog produces the same solution $Y^{\ast}$ as the minimization of $\widetilde{F}_{i}$
\[
\argmin_{Y \in \Stnp} F_i(Y) = \argmin_{\widetilde Y \in \Rnp} \widetilde F_i(\widetilde Y) = Y^{\ast},
\]
with $Y^{\ast}$ the Riemannian center of mass on $\Stnp$ of $X_{i-1}$ and $X_{i+1}$.
\end{lemma} 
\begin{proof}
Since $d(X_{i-1},X_{i+1})\leq \delta_g$, Lemma~\ref{lemma:geodesic_exists_for_close_neighbours} gives that the minimizer of $F_i$ on $\Stnp$ is unique and equals the Riemannian center of mass $Y^{\ast}$. To show that it also equals the minimizer of $\widetilde F_i$ on $\Rnp$, take any $\widetilde Y \in \Rnp$. If $\sigma_k(\widetilde Y) > 0$, then we can write
\[
 \widetilde Y = Y + \Delta, \qquad Y = \PSt \widetilde Y \in \Stnp.
\]
Using that $Y^{\ast}$ is the minimizer of $F_i$ on $\Stnp$, we thus get
\[
 \widetilde F_i(\widetilde Y) = d^{2}(X_{i-1}, Y) + d^{2}(Y,X_{i+1}) + 2 \| \Delta \|_{\F}^{2} \geq F_i(Y) \geq F_i(Y^{\ast}).
\]
The same inequality holds trivially if $\sigma_k(\widetilde Y) = 0$ since then $\widetilde F_i(\widetilde Y) = +\infty$. Finally, since $\widetilde F_i(Y^{\ast}) = F_i(Y^{\ast})$, we obtain that $\widetilde F_i$ is also uniquely minimized by $Y^{\ast}$.   
\end{proof}

\begin{lemma}\label{lem:leapfrog_one_step_is_GS}
Suppose that for all iterations $k=0,1,\ldots$, the iterates of leapfrog satisfy
\[
   d(X_{i-1}^{(k)}, X_{i+1}^{(k-1)}) \leq \delta_g,
\]
for all $ i = 1, 2, \ldots, m-2 $. Then, the leapfrog algorithm started in $X^{(0)}$ generates the same iterates as the nonlinear Gauss--Seidel algorithm started in $X^{(0)}$ and applied to 
\[
 \min_{X_1,\ldots, X_{m-2} \in \Rnp} \widetilde F(X_{1}, \ldots, X_{m-2} ).
\]
\end{lemma}
\begin{proof}
By induction. Suppose true until substep $i-1$ of iteration $k$. Then, leapfrog computes the new iterate as
\[
 X_i^{(k)} = \argmin_{Y \in \Stnp} d^2(X_{i-1}^{(k)},Y) + d^2(Y,X_{i+1}^{(k-1)}).
\]
The uniqueness of the minimizer follows from Lemma \ref{lemma:geodesic_exists_for_close_neighbours} and $d(X_{i-1}^{(k)}, X_{i+1}^{(k-1)}) \leq \delta_g$. Likewise, nonlinear Gauss--Seidel computes
\[
 \widetilde X_i^{(k)} = \argmin_{\widetilde Y \in \Rnp}  \widetilde F(X_1^{(k)}, \ldots, X_{i-1}^{(k)}, \widetilde Y, X_{i+1}^{(k-1)}, \ldots, X_{m-2}^{(k-1)}),
\]
and the uniqueness of the minimizer follows from our reasoning below. 
Both minimization problems are the same as minimizing $F_i$ and $\widetilde F_i$ from Lemma~\ref{lemma:substep_leapfrog_equals_substep_GS} but with $X_{i-1}^{(k)}$ and $X_{i+1}^{(k-1)}$ taking the roles of $X_{i-1}$ and $X_{i+1}$, respectively. By Lemma~\ref{lemma:substep_leapfrog_equals_substep_GS}, the minimizers of both problems are the same and hence $X_i^{(k)} = \widetilde X_i^{(k)}$. The above reasoning can also be applied to the base case $k=i=1$ since $X^{(1)}_0 = X^{(0)}_0$. Hence, we have proven the result.  
\end{proof}

If the initial points are close enough, the iterates in leapfrog stay close.
\begin{lemma}\label{lem:leapfrog_iterates_stay_close}
Let $X^{(0)} \in \Stnp^{m}$ be such that $d(X_{i-1}^{(0)}, X_{i}^{(0)}) \leq \tfrac{1}{2}\delta_g$ for all $1 \leq i \leq m -1 $. Then, leapfrog started at $X^{(0)}$ is well defined and all its iterates $X^{(k)}$ satisfy for all $1 \leq i \leq m-2$ and $k\geq 1$
\begin{equation}\label{eq:closer_points_during_leapfrog}
d(X_{i-1}^{(k)}, X_{i}^{(k)}) = d(X_{i}^{(k)}, X_{i+1}^{(k-1)}) \leq \tfrac{1}{2}\delta_g.
\end{equation}
\end{lemma} 
\begin{proof}
By induction. Suppose true for all substeps $i$ until iteration $k-1$ and until substep $i-1$ of iteration $k$. This implies in particular
\[
 d(X_{i-1}^{(k)}, X_{i}^{(k-1)}) \leq \tfrac{1}{2}\delta_g , \quad d(X_{i}^{(k-1)}, X_{i+1}^{(k-1)}) \leq \tfrac{1}{2} \delta_g.
\]
By triangle inequality for the Riemannian distance,
\[
 d(X_{i-1}^{(k)}, X_{i+1}^{(k-1)}) \leq d(X_{i-1}^{(k)}, X_{i}^{(k-1)}) + d(X_{i}^{(k-1)}, X_{i+1}^{(k-1)}) \leq \delta_g,
 \]
 Lemma~\ref{lemma:geodesic_exists_for_close_neighbours} gives that the leapfrog iteration is well defined and produces the unique minimizer
\[
 X_i^{(k)} = \argmin_{Y \in \Stnp} d^2(X_{i-1}^{(k)},Y) + d^2(Y,X_{i+1}^{(k-1)}).
\]
We thus have
\[
 d^2(X_{i-1}^{(k)},X_i^{(k)}) + d^2(X_i^{(k)},X_{i+1}^{(k-1)}) \leq d^2(X_{i-1}^{(k)},X_i^{(k-1)}) + d^2(X_i^{(k-1)},X_{i+1}^{(k-1)}) \leq \tfrac{1}{2} \delta_g^2.
\]
Since $X_i^{(k)}$ is the midpoint of the geodesic connecting $X_{i-1}^{(k)}$ to $X_{i+1}^{(k-1)}$, we also have
\[
  d(X_{i-1}^{(k)},X_i^{(k)}) = d(X_i^{(k)},X_{i+1}^{(k-1)}).
\]
Combining these two results proves~\eqref{eq:closer_points_during_leapfrog} until substep $i$ at iteration $k$. Since $X_0^{(k+1)} = X_0^{(k)} = X_0^{(0)}$, the case for substep $i=1$ and iteration $k+1$ satisfies the same reasoning as above. The same is true for the base case $i=k=1$, which ends the proof.  
\end{proof}

Hence, combining Lemmas~\ref{lem:leapfrog_one_step_is_GS} and~\ref{lem:leapfrog_iterates_stay_close}, we get our desired result:

\begin{proposition}\label{prop:leapfrog_is_GS}
Let $X^{(0)} \in \Stnp^{m}$ be such that $d(X_{i-1}^{(0)}, X_{i}^{(0)}) \leq \tfrac{1}{2}\delta_g$ for all $1 \leq i \leq m$. Then the leapfrog algorithm applied to $F$ is equivalent to the nonlinear Gauss--Seidel method applied to $\widetilde F$.
\end{proposition}

We can now proceed and analyze the convergence of this nonlinear Gauss--Seidel method using standard theory.

\subsection{First-order optimality}
From Prop.~\ref{prop:leapfrog_is_GS}, we know that at iteration $k \geq 1$ and for subinterval $i \in \{ 1,\ldots, m-2\}$, leapfrog solves the following unconstrained optimization problem 
\begin{equation*}
   \min_{X_{i} \in \Rnp} \widetilde F_{i}^{k}(X_{i}),
\end{equation*}
where the objective function is defined as
\begin{equation*}
    \widetilde F_i^{k}(Y) = \widetilde d^2(X_{i-1}^{(k)}, Y) + \widetilde d^2(Y,X_{i+1}^{(k-1)}).
\end{equation*}
Recall that $X_{i-1}^{(k)}, X_{i+1}^{(k-1)} \in \Stnp$ are the neighboring points of $X_i$ and that $X_{i-1}^{(k)}$ was previously updated and that $X_{i+1}^{(k-1)}$ will be updated next.

Let us define
\begin{equation*}
    \cG_{i}(Y) = \nabla_{Y}  \widetilde F_{i}^{k}(Y) = \nabla_{Y} \widetilde{d}^{2}(X_{i-1}^{(k)},Y) + \nabla_{Y} \widetilde{d}^{2} (X_{i+1}^{(k-1)}, Y).
\end{equation*}
At the minimizer $X_i$, the gradient of $\widetilde F_i^{k}$ vanishes, i.e., $ \cG_{i}(X_{i}) = 0 $. Likewise, if we take all the minimizers $X = (X_1, \ldots, X_{m-2})$ together, they will satisfy 
\begin{equation*}
    \begin{cases}
         \cG_{1}(X_{}) = \nabla_{X_{1}} \widetilde{d}^{2}(X_{0},X_{1}) + \nabla_{X_{1}} \widetilde{d}^{2} (X_{1}, X_{2}) = 0, \\
         \cG_{2}(X_{}) =\nabla_{X_{2}} \widetilde{d}^{2}(X_{1},X_{2}) + \nabla_{X_{2}} \widetilde{d}^{2} (X_{2}, X_{3}) = 0, \\
         \qquad \vdots \\
         \cG_{m-2}(X_{}) =\nabla_{X_{m-2}} \widetilde{d}^{2}(X_{m-3},X_{m-2}) + \nabla_{X_{m-2}} \widetilde{d}^{2} (X_{m-2}, X_{m-1}) = 0.
    \end{cases}
\end{equation*}
This can be written compactly as $\cG(X)=0$, where $\cG$ is defined componentwise $ \cG_{i} \colon
\Rnp \to \Rnp $, for $i = 1,\ldots, m-2$.

\subsection{Known results on local convergence}
Assuming convergence to the limit point $X_1^{\ast}, X_2^{\ast}, \ldots, X_{m-2}^{\ast}$, the asymptotic convergence rate is determined by the spectral radius of a certain blockwise partitioning of the Hessian of $\widetilde F$ at this limit point.

\begin{theorem}[Nonlinear block Gauss--Seidel theorem]\label{thm:asymptotic_speed_nonlinearBGS}
Let $ \cG \colon \cD \subset \R^{(m-2)np} \to \R^{(m-2)np} $ be continuously differentiable in an open neighborhood $ \cB_{0} \subset \cD $ of a point $ X^{\ast} \in \cD $ for which $ \cG(X^{\ast}) = 0 $. Consider the decomposition of $ \cG' = D - L - U $ into its block diagonal, strictly lower-, and strictly upper-triangular parts, and suppose that $ D(X^{\ast} ) $ is nonsingular and $ \rho( M^{\mathrm{BGS}}(X^{\ast} ) ) < 1 $, where $ M^{\mathrm{BGS}} = (D-L)^{-1} U $. Then there exists an open ball $ \cB = \cB(X^{\ast}, \delta ) $ in $ \cB_{0} $ such that, for any $ X^{(0)} \in \cB $, there is a unique sequence $ \lbrace X^{(k)} \rbrace \subset \cB $ which satisfies the nonlinear Gauss--Seidel prescription. Moreover, $ \lim_{k\to\infty} X^{(k)} = X^{\ast} $ and for any $ X^{(0)} \in \cB_{0} $, the convergence rate in the form $ \limsup_{k\to \infty} \sqrt[k]{\| X^{(k)} - X^{\ast} \|} $ is upper bounded by $ \rho( M^{\mathrm{BGS}}(X^{\ast} ) ) $.
\end{theorem}

\begin{proof}
   As a direct extension of \cite[Theorem~10.3.5]{Ortega:2000}.  
\end{proof}

This theorem shows the need for the Hessian of $ \widetilde{F} $ (i.e., $ \cG' $) and its block $ D-L-U $ decomposition. As we shall see, our matrix $ \cG' $ is given by the sum of two matrices $ \cG' = A + E $, where $ A $ is symmetric block tridiagonal and positive definite, and $ E $ can be regarded as a perturbation matrix. Since it is very difficult to compute the spectral radius of $ M^{\mathrm{BGS}} $ with this perturbation $ E $, we will not use Theorem~\ref{thm:asymptotic_speed_nonlinearBGS} directly. Instead, we will use the \emph{Householder--John theorem} \cite[Corollary~3.42]{Hackbusch:2016aa}, which states that if $ \cG' $ is positive definite, then the $ M^{\mathrm{BGS}} $ from Theorem \ref{thm:asymptotic_speed_nonlinearBGS} satisfies $ \rho(M^{\mathrm{BGS}}) < 1 $.
In other words, (linear) block Gauss--Seidel for a symmetric and positive definite $ \cG' $ always converges monotonically in the energy norm \cite[Theorem~3.53]{Hackbusch:2016aa}. Therefore, we only need to restrict the perturbation $ E $ such that the whole matrix $ \cG' $ is symmetric and positive definite. In order to do that, we will also use a block version of the Gershgorin circle theorem \cite[Theorem~2]{Feingold:1962aa}.

\subsection{Local convergence}\label{sec:local_convergence}
As required in Theorem~\ref{thm:asymptotic_speed_nonlinearBGS}, we compute the Hessian as the Jacobian matrix $\cG'(X)$, a square matrix of size $(m-2)np$. By symmetry of the Hessian, we can write this compactly as
\begin{equation*}
\cG' = \begin{bmatrix}
        D_{10} + D_{12} &  L_{12}\tr  &   &  & \\
        L_{12} &  D_{21} + D_{23}  &  L_{23}\tr &  & \\
          & \ddots  &  \ddots & \ddots \\
          &  & L_{m-3,m-2} & D_{m-2,m-3} + D_{m-2,m-1}
     \end{bmatrix},
\end{equation*} 
where 
\[
 L_{ij} = \nabla_{X_{i}} \nabla_{X_{j}} \widetilde{d}^{2}(X_{i},X_{j}) \qquad \text{and} \qquad
 D_{ij} = \nabla^2_{X_{i}} \widetilde{d}^{2}(X_{i},X_{j}) 
\] 
denote the mixed and double derivatives\footnote{Observe that $L_{ij} = L_{ji}\tr$ by equality of mixed derivatives but in general $D_{ij} \neq D_{ji}\tr$ since only the variable corresponding to the first index is derived.}.

We now turn to the computation of these derivatives $L_{ij}$ and $D_{ij}$. To that end, the following lemma is convenient since it writes $\widetilde{d}^{2}(X_{i},X_{j})$ as an expansion that does not explicitly use the Riemannian distance.

\begin{lemma}\label{lemma:expansion_ext_dist_function}
Let $\widetilde{X}, \widetilde{Y} \in \Rnp $ such that $\sigma_p(\widetilde X) > 0$ and $\sigma_p(\widetilde Y) > 0$, then
\begin{equation}\label{eq:expansion_ext_dist_function}
\begin{split}
   \widetilde{d}^2(\widetilde{X}, \widetilde{Y})  = & \ \| \PSt\widetilde{X} - \PSt\widetilde{Y} \|^{2}_{\F} - \tfrac{1}{2} \| I_{p} - \big(\PSt\widetilde{X}\big)\tr\PSt\widetilde{Y}\|^{2}_{\F} \\
   & + \| \widetilde{X} - \PSt\widetilde{X} \|_{\F}^{2} + \| \widetilde{Y} - \PSt\widetilde{Y} \|_{\F}^{2} + \cO (\| \PSt\widetilde{X} - \PSt\widetilde{Y} \|^{4}_{\F}).
\end{split}
\end{equation}
\end{lemma}

\begin{proof} See App.~\ref{app:proof_expansion_ext_dist_function}.  \end{proof}

In the following, denote $\delta_{ij} = \| X_i - X_j \|_2$ for any $X_i, X_j \in \Stnp$.

\begin{lemma}\label{lemma:hessian_at_X_Z_on_Stiefel}
Let $ X_{i} \in \Stnp $. Then
\begin{align}
    D_{ij} &= 2 I_{np} + \tfrac{1}{2}\,(
    X_{i}\tr  \otimes X_{i} ) \, \Pi_{p,n} - \tfrac{1}{2}\,(I_{p} \otimes
    X_{i}X_{i}\tr  ) + \Delta_{ij},  \label{eq:Dij_on_Stiefel} \\
    L_{ij} &= - 2I_{np} + \tfrac{1}{2}(X_i\tr \otimes X_i) \, \Pi_{p,n} + \tfrac{3}{2}(I_{p}\otimes X_iX_i\tr) + \Lambda_{ij}, \label{eq:Lij_on_Stiefel}
\end{align}    
with $ \| \Delta_{ij} \|_2 \leq 14 \delta_{ij} + 10 \delta_{ij}^2 $ and $\| \Lambda_{ij} \|_2 \leq \tfrac{11}{2} \delta_{ij} + 10 \delta_{ij}^2 + 4 \delta_{ij}^3$. Here, $\Pi_{p,n}$ is the vec-permutation matrix defined as the matrix that satisfies $\vecop (X) = \Pi_{n,p} \vecop (X\tr)$; see, e.g., \cite[Eq.~(5)]{Henderson:1981}.
\end{lemma}

\begin{proof} See App.~\ref{app:proof_hessian_at_X_Z_on_Stiefel}.  \end{proof}

Our aim is to diagonalize $\cG'$. We will do this in a few steps. First, observe that $\cG'$ remains block tridiagonal if it is transformed using a compatible block diagonal matrix $\cQ = \diag\lbrace Q_1, Q_2, \ldots, Q_{m-2} \rbrace $:

\begin{equation}\label{eq:block_QGQ}
  \cQ\tr \cG' \cQ = \begin{bmatrix}
        Q_1\tr (D_{10} + D_{12}) Q_1  &  Q_1\tr L_{12}\tr Q_2   &   &  & \\
        Q_2\tr L_{12} Q_1 &  Q_2\tr (D_{21} + D_{23})Q_2   &  Q_2\tr L_{23}\tr Q_3 &  & \\
         & \ddots  &  \ddots & \ddots \\
     \end{bmatrix},
\end{equation}

Here, the $Q_1, \ldots, Q_{m-2} \in \R^{np \times np}$ can be any orthogonal matrices. The lemma below shows us how to choose these matrices so that we obtain diagonal blocks in $\cQ\tr \cG' \cQ$, up to first order in $\delta_{ij}$.  

\begin{lemma}\label{lemma:diagonalization_hessian_on_Stiefel}
Let $X_i^\perp \in \R^{n \times (n-p)}$ be such that $X_i\tr X_i^{\perp} = O_{p \times (n-p)} $ and $(X_i^{\perp})\tr X_i^{\perp} = I_{(n-p)}$. Define the orthogonal matrices
\begin{equation*}
    \overbar{Q}_i = \big[ I_{p}\otimes X_i \quad I_{p}\otimes X_i^{\perp} \big],    
\end{equation*}
and similarly for~$ \overbar{Q}_{j}$. Then, there exists an orthogonal matrix $\widehat{Q}$, only depending on $n$ and $p$, 
such that $ Q_i = \bar{Q}_i \widehat{Q} $ and $ Q_j = \bar{Q}_j \widehat{Q} $ satisfy
\begin{align}
    \| Q_i \tr D_{ij} Q_i - D   \|_2 &\leq C_D^{(ij)}, &  D &= \diag\left\lbrace I_{p(p-1)/2}, \ 2 \, I_{np-p(p-1)/2} \right\rbrace, \label{eq:diag_Dij_on_Stiefel}  \\
\|Q_j \tr L_{ij} Q_i -  L  \|_2 &\leq C_L^{(ij)}, &    L &= \diag\left\lbrace -I_{p(p-1)/2}, \  -2 I_{(n-p)p}, \  O_{p(p+1)/2} \right\rbrace, \label{eq:diag_Lij_on_Stiefel} 
\end{align}
where $C_D^{(ij)} = 14 \delta_{ij} + 10 \delta_{ij}^2$ and $C_L^{(ij)} =  \tfrac{15}{2} \delta_{ij} + \tfrac{31}{2} \delta_{ij}^2 + 14 \delta_{ij}^3 + 4 \delta_{ij}^4$.
\end{lemma}

\begin{proof} See App.~\ref{app:proof_diag_Hij}.  \end{proof}

The matrix $\widehat{Q}$ above is related to the diagonalization of the vec-permutation matrix $\Pi_{p,p}$; see~\eqref{eq:def_hat_Pi} in App.~\ref{app:proof_diag_Hij} for its definition. 
It is therefore also independent of $X_i$. This is a crucial property to obtain the following result.

\begin{lemma}\label{lemma:condition_for_pos_def_Gprime}
Define $ \delta = \max_{0 \leq i \leq m-2} \delta_{i,i+1} $ and assume $ \delta \leq 1 $. Then the minimal eigenvalue of $\cG'$ is bounded by
\[ 
   \lambda_{\min}(\cG') \geq 2 - 2 \cos\tfrac{\pi}{m-1} - 43 \delta - 90 \delta^{2}.
\]
As a consequence, $\cG'$ is symmetric and positive definite when
\[
   \delta < \frac{1}{180}\left( \sqrt{2\,569 - 720\cos\tfrac{\pi}{m - 1}} - 43 \right).
\]
\end{lemma}



\begin{proof}
From Lemma~\ref{lemma:diagonalization_hessian_on_Stiefel}, recall the diagonal matrices $D$ and $L$, and the orthogonal matrices $Q_1, \ldots, Q_{m-2}$. Define $\cQ = \diag\lbrace Q_1, Q_2, \ldots, Q_{m-2} \rbrace $. Substituting the nonzero blocks in~\eqref{eq:block_QGQ} by
\[
 Q_i\tr(D_{i,i-1} + D_{i,i+1}) Q_i  =  2D + E_{ii} , \qquad
 Q_{i+1}\tr L_{i,i+1} Q_i  = L + E_{i,i+1},
 \]
we can write $\cQ\tr \cG' \cQ $ as 
\begin{equation}\label{eq:block_QGQ_diag}
  \cQ\tr \cG' \cQ =  \begin{bmatrix}
        2D  &  L   &   &  & \\
        L &  2D   &  L &  & \\
          & \ddots  &  \ddots & \ddots \\         
     \end{bmatrix} +
     \begin{bmatrix}
        E_{11}  &  E_{12}\tr   &   &  & \\
        E_{12} &  E_{22}   &  E_{23}\tr &  & \\
          & \ddots  &  \ddots & \ddots \\         
     \end{bmatrix} \eqqcolon A + E.
\end{equation}
Eq. \eqref{eq:block_QGQ_diag} is an approximate tridiagonalization of the matrix $\cG'$.
Observe that the symmetric matrices $A$ and $E$ have compatible block partitioning. Furthermore, from Lemma~\ref{lemma:diagonalization_hessian_on_Stiefel}, we get immediately that
\[
 \| E_{ii} \|_2 \leq 28 \delta + 20 \delta^2 \eqqcolon C_D, \qquad
 \| E_{i,i+1} \|_2 \leq \tfrac{15}{2} \delta + \tfrac{31}{2} \delta^2 + 14 \delta^3 + 4 \delta^4 \eqqcolon C_L.
\]

We will regard $\cQ\tr \cG' \cQ$ as an $\cO (\delta)$ perturbation of $A$. Using the properties of Kronecker products, we can write
\begin{equation}\label{eq:cQtrcGcQ}
    A = 2I_{m-2} \otimes D + M \otimes L, \qquad M = \begin{bmatrix}
        0 & 1 &  &  \\
        1 & \ddots & \ddots &  \\
         &  \ddots  & \ddots  & 1 \\
         &  & 1 & 0
    \end{bmatrix} \in \R^{(m-2) \times (m-2)}.
\end{equation}
Thanks to the Kronecker structure in~\eqref{eq:cQtrcGcQ} and the diagonal matrices $D$ and $L$, the eigenvalues of $A$ are easily determined as
\[
 \lambda_{jk} = 2 d_j + \mu_k \ell_j, \quad j=1,\ldots,np, \quad k=1,\ldots, m-2,
\]
where $d_j$ and $\ell_j$ are the diagonal entries of $D$ and $L$, respectively, and $\mu_k$ are the eigenvalues of the Toeplitz matrix $M$. Using~\cite[Eq.~(2.7)]{Gover:1994}, we find
\[
 \mu_k = -2 \cos \tfrac{k \pi}{m-1}, \qquad k=1,\ldots, m-2.
\]
Together with~\eqref{eq:diag_Dij_on_Stiefel} and~\eqref{eq:diag_Lij_on_Stiefel}, this allows us to determine that the minimal value among all $\lambda_{jk}$ corresponds to $ j = 1 $ and $ k = m-2 $. We thus obtain
\[
 \lambda_{\min}(A) = 2 - 2 \cos\tfrac{\pi}{m-1} > 0 \quad \text{for all $m \geq 2$}.
\]
By Weyl's inequality \cite[Corollary~4.9]{Stewart:1990}, $\lambda_{\min}(\cG') = \lambda_{\min}(A+E)> 0$ is guaranteed if $\|E\|_2 < \lambda_{\min}(A)$. To bound $\|E\|_2$, we use a block version of the Gershgorin circle theorem (see \cite[Theorem~2]{Feingold:1962aa} and also~\cite[Remark~1.13.2]{Tretter:2008aa}). Applied to the symmetric block tridiagonal matrix $E$, it guarantees that its eigenvalues are included in the union of intervals 
\[
 \bigcup_{i=1}^{m-2} \bigcup_{k=1}^{np} [\varepsilon^{(i)}_k - R_i, \varepsilon^{(i)}_k + R_i], \qquad R_i = \|E_{i-1,i}\|_2 + \|E_{i,i+1}\tr\|_2 \leq 2 C_L,
\]
where $\varepsilon^{(i)}_k$ is the $k$th eigenvalue of $E_{ii}$. These eigenvalues $\varepsilon^{(i)}_k$ are all bounded in magnitude by $C_D$. Hence $\|E\|_2 \leq C_D + 2 C_L = 43 \delta + 51 \delta^2 + 28 \delta^3 + 8 \delta^4$. Since $ \delta<1 $, it is easily verified that $\|E\|_2 \leq 43 \delta + 90 \delta^2 $ and thus the matrix $\cG'$ remains positive definite if $ 43 \delta + 90 \delta^2 < \lambda_{\min}(A) $, i.e.,
\[
   \delta < \frac{1}{180}\left( \sqrt{2\,569 - 720\cos\tfrac{\pi}{m - 1}} - 43 \right). 
\]
\end{proof}

All put together, we have the final result of local convergence.

\begin{theorem}
If the leapfrog algorithm is started with $ \delta $ satisfying the condition of Lemma~\ref{lemma:condition_for_pos_def_Gprime}, then it converges to the unique minimizing geodesic connecting $X_0$ and $X_{m-1}$, provided that the initial intermediate points are sufficiently close to that geodesic.
\end{theorem}

\begin{proof}
We use \cite[Corollary~3.42]{Hackbusch:2016aa} which 
states that if $ \cG' $ is positive definite and can be split into the sum of an arbitrary positive definite matrix and an arbitrary symmetric matrix, then the scalar Gauss--Seidel converges, i.e., $ \rho(M^{\mathrm{BGS}}) < 1 $, and the convergence is monotone with respect to the energy norm $ \| \cdot \|_{\cG'} $. By~\cite[Theorem~3.53]{Hackbusch:2016aa}, we know that this theorem remains valid for any block version. 

Now, the splitting \eqref{eq:block_QGQ_diag} has exactly the form prescribed by~\cite[Corollary~3.42]{Hackbusch:2016aa}, because $ A $ is positive definite and $ E $ is symmetric.
By Lemma~\ref{lemma:condition_for_pos_def_Gprime}, we know that $ \cG' $ remains positive definite if $ \delta < \frac{1}{180}\left( \sqrt{2\,569 - 720\cos\tfrac{\pi}{m - 1}} - 43 \right) $. Under these conditions, the leapfrog algorithm converges as a block Gauss--Seidel method to the minimizing geodesic connecting $X_{0}$ and $X_{m-1}$.  
\end{proof}

\section{Some observations and open problems}\label{sec:open_problems}
For $m$ large, Lemma~\ref{lemma:condition_for_pos_def_Gprime} gives that $\cG'$ is positive definite when $ \delta \lesssim \pi^2 / 43 m^2 $. Let $ d_{0} = \|X_0 - X_{m-1}\|_{2} $ be the distance between the two endpoints. Then by equidistant partitioning of the intermediate points, one has $\delta \simeq d_{0} / m$. To guarantee a positive definite $\cG'$, we would then need $d_{0} / m  \lesssim \pi^2 / 43 m^2$ which implies $m \lesssim 0.23 / d_{0}$.

This result is unsatisfactory since it would have been desirable to guarantee positive definiteness of $\cQ\tr \cG' \cQ =A+E$ with orthogonal $\cQ$ by increasing the number of points $ m $ given a fixed $ d_{0} $. Unfortunately, we cannot guarantee this with our proof. The problem is that $\|E\|_2 = \cO (\delta)$ whereas $\lambda_{\min}(A) = \cO (1/m^2)$, which leads to our condition that $ m $ needed to be smaller than some fixed fraction of the original distance $d_0$. If $\|E\|_2 = \cO (\delta^2)$, then there would be no condition on $m$ since $\delta^2 \simeq  d_0^2 / m^2 \lesssim 1/m^2$ is sufficient to guarantee $\lambda_{\min}(A+E)>0$. However, it would still not be satisfactory since the perturbation does not lead to an improvement with increasing $ m $, for which one probably needs $\|E\|_2 = \cO (\delta^3)$. As we show below, there is a strong numerical indication that this is not the case with our choice of extended distance function. 

Numerical experiments reported in Figure \ref{fig:error_plot_lambdas_limit_false} suggest that the minimal eigenvalues of $\cG'$ and $A$ differ by $\cO ( \delta^2 )$, whereas our perturbation analysis only showed $ \|E\|_2= \cO ( \delta )$. It is however not trivial to prove this result. 
Indeed, up to first order, we can study the eigenvalues of the symmetric matrix $A+E$ by using the derivative formula~\cite[Theorem~2.3]{Stewart:1990}
\begin{equation}\label{eq:eigenval_perturb_formula}
   \lambda_{\min} (A + E) = \lambda_{\min} (A) + v_{\min}\tr E v_{\min} + \cO ( \| E \|^2 ),
\end{equation}
where $\lambda_{\min} (A) $ is assumed to be isolated (as is indeed the case) and $ v_{\min} $ is its associated eigenvector. One possibility to improve on our bounds, at least asymptotically, would be to prove that $ \vert v_{\min}\tr E v_{\min} \vert= \cO ( \delta^{3} ) $.
However, in the same figure, $\vert v_{\min}\tr E v_{\min} \vert $ seems to be again $\cO (\delta^2)$. In addition, all these conclusions remain true in the limiting geodesic.

\begin{figure}
\centering
\begin{minipage}{.48\textwidth}
  \centering
  \includegraphics[height=.725\linewidth]{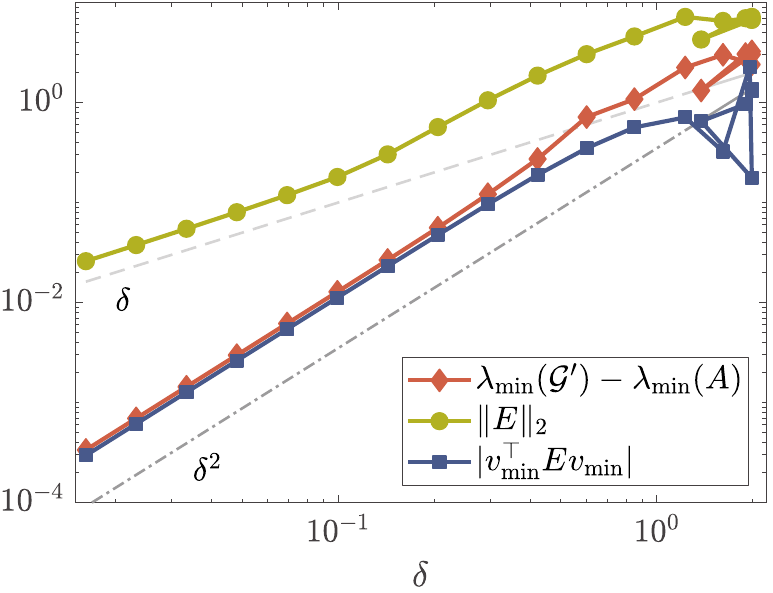}
  \caption{Eigenvalue perturbations -- not at the limiting geodesic.}\label{fig:error_plot_lambdas_limit_false}
\end{minipage}
\hfill
\begin{minipage}{.48\textwidth}
  \centering
  \includegraphics[height=.725\linewidth]{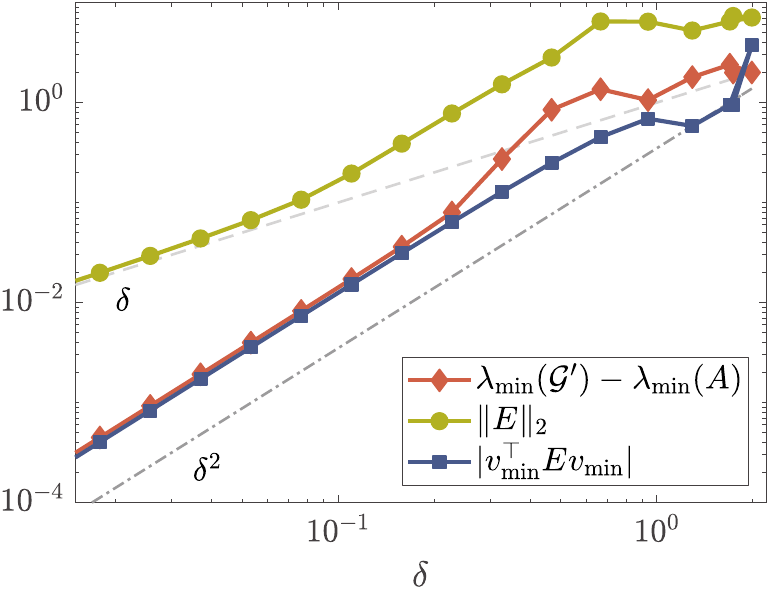}
  \caption{Eigenvalue perturbations -- at the limiting geodesic.}\label{fig:error_plot_lambdas_limit_true}
\end{minipage}
\end{figure}

Another problem with the matrix $ A $ and $ \cG' $ is that it has a bad spectral gap $ \gamma $ (i.e., the difference of smallest and second smallest eigenvalue) when $ m $ grows. Numerical observations suggest that the spectral gap might be $ \cO ( 1/m ) $ which complicates non-asymptotic bounds.

As a last remark, one could resort to a more general theory for the convergence of nonlinear block Gauss--Seidel for a quasi-convex objective function \cite{Grippo:2000}, which requires quasi-convexity for each $ X_i $ alone.
Looking at the Hessian $ \cG' $ where all $X_j$ except $ X_i $ are constant, the only block that is left in the matrix $ \cG' $ is the diagonal one, namely $ D_{i,i-1} + D_{i,i+1} $.
Using Lemma~\ref{lemma:diagonalization_hessian_on_Stiefel}, we immediately get the eigenvalues of this block. Now, for $ C_D^{(ij)} < 1 $ in~\eqref{eq:diag_Dij_on_Stiefel} 
we get strong convexity in $ X_i $ alone.
One problem with this approach is that the feasible set has to be a Cartesian product of convex subsets of $ \Rnp $.
Moreover, the result in \cite{Grippo:2000} only guarantees subsequence convergence, and there is no rate of convergence or contraction rate for the whole sequence. Hence the convergence behavior could also be slower than linear.

\section{Numerical experiments}\label{sec:stiefel_experiments}
The purpose of the numerical experiments in this section is to demonstrate that when $ m \to \infty $, leapfrog still converges, although very slowly. 
We conducted our experiments on a laptop Lenovo ThinkPad T460s with Ubuntu 20.04 LTS and MATLAB R2018a installed, with Intel Core i7-6600 CPU, 20GB RAM, and Mesa Intel HD Graphics 520.

As a concrete example to demonstrate the leapfrog algorithm, let us consider the Stiefel manifold $ \mathrm{St}(12,3) $. We fix one point $ X = [ I_{3} \ \ O_{9 \times 3} ]\tr $, while the other point $ Y $ is placed at the distance $ L^{\ast} = 0.96\,\pi $ from $ X $. This is done by creating a tangent vector to $ \mathrm{St}(12,3) $ at $ X $ of length $ L^{\ast} $, and then mapping it to $ \mathrm{St}(12,3) $ via the Riemannian exponential~\eqref{eq:closed-form-sol-geodesic}.  We want to recover this distance using the leapfrog algorithm and study its convergence.
Even though $ n = 12 $, $ p = 3 $ are small, $ \mathrm{St}(12,3) $ is an interesting case because with $ X = [ I_{3} \ \ O_{9 \times 3} ]\tr $ and $ L^{\ast} = 0.96 \, \pi $ the single shooting method does not converge. The algorithm in \protect{\cite{Zimmermann:2017}} also does not work for this example, due to the fact that the principal matrix logarithm is not defined.
On the contrary, leapfrog with $ m = 4 $ converges and recovers the exact distance.

In the first experiment, for each value of $m \in  \{ 10,20,50,100 \}$, we construct an initial guess $X^{(0)}$ by placing $ m-2 $ intermediate points randomly along the linear segment connecting $ X $ and $ Y $ in the embedding space and projecting them to the Stiefel manifold.
We then apply leapfrog for 300 iterations and monitor the convergence behavior of
\[
   \textrm{err-$k$} = \| X^{(k)} - X^{\ast} \|_{\F},
\]
where $X^{\ast}$ is the solution of leapfrog (i.e., a uniformly distributed tuple corresponding to the global geodesic that was constructed above), and $ X^{(k)} $ is the approximate solution at iteration $ k $ of leapfrog.
This is illustrated in Figure~\ref{fig:Convg_LF_12_3_m10_100_maxiter_300}, from which it is clear that for large $ m $ leapfrog always converges, albeit very slowly.
Although in practice there is limited use of leapfrog on the Stiefel manifold with many points, here it is considered as a model case to show that convergence deteriorates as $ m \to \infty $.

\begin{figure}
\centering
\begin{minipage}{.48\textwidth}
  \centering
  \includegraphics[height=.725\linewidth]{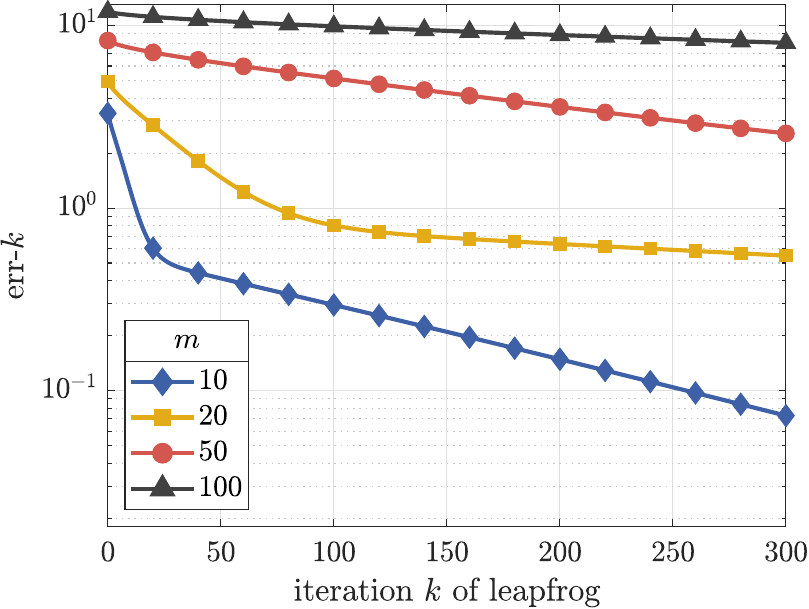}
  \caption{Convergence behavior of err-$k$ for increasing values of $m$.}\label{fig:Convg_LF_12_3_m10_100_maxiter_300}
\end{minipage}
\hfill
\begin{minipage}{.48\textwidth}
  \centering
  \includegraphics[height=.725\linewidth]{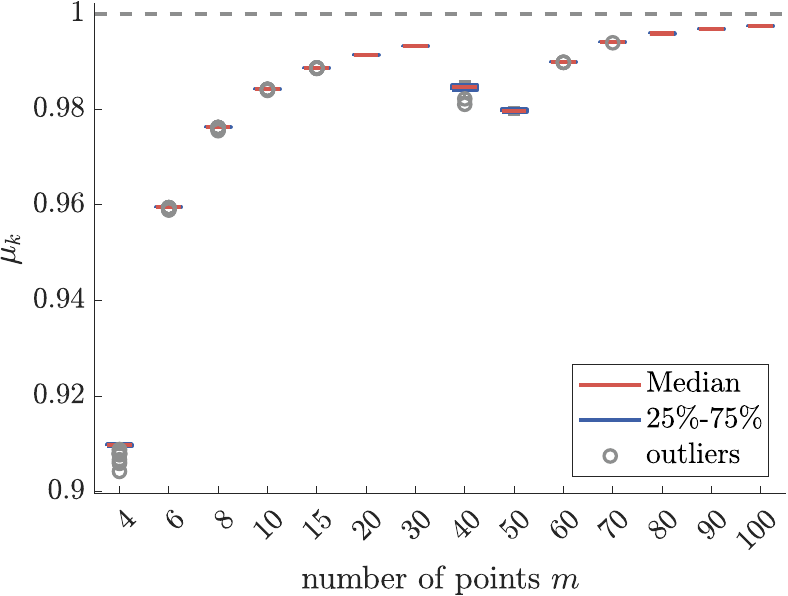}
  \caption{Boxplot of $ \max_{k} \lbrace \mu_{k}^{(i)} \rbrace $ for increasing values of $m$.}\label{fig:boxplot}
\end{minipage}
\end{figure}

In the second experiment, for each $ m \in  \{ 4,6,8,10,\ldots, 100 \} $, we perform 50 leapfrog iterations and repeat this experiment for 100 random initializations of the initial guess $X^{(0)}$. Again, the choice of $m \in \{ 4,6,8,10,\ldots, 100 \} $ is made in order to illustrate the theory, in particular, the deterioration of convergence as $m$ grows.

For each experiment $i$, we define the error reduction rate\footnote{In the limit $ k \to \infty $, this gives the asymptotic Q-rate of convergence of the sequence.} as
\[
   \mu_{k}^{(i)} = \frac{\textrm{err-$(k+1)$}}{\textrm{err-$k$}}, \quad \text{for} \quad k = 0, 1, \ldots, 49, \quad i=1,\ldots,100,
\]
and we compute the worst and the median reduction rates across all the experiments, namely,
   $\max_{i,k} \{ \mu_{k}^{(i)} \}$ and
   $\med_{i} \max_{k} \{ \mu_{k}^{(i)} \}$. 
Since during the first iterations leapfrog is faster, we also compute the convergence factor given by
   $\max_{i} \{ \mu_{0}^{(i)} \}$. 

From Table~\ref{tab:1}, we see that the convergence of leapfrog deteriorates as $ m $ increases but it remains strictly smaller than $1$. For small values of $ m $, $ \max_{i} \lbrace \mu_{0}^{(i)} \rbrace $ and $ \max_{i,k} \lbrace \mu_{k}^{(i)} \rbrace $ are significantly different, whereas for large values of $ m $, they are quite similar. The same conclusion can be reached from Figure~\ref{fig:boxplot} where boxplots show the dispersion and skewness in the $ \mu^{(i)}_k $. Clearly, the convergence factors become very concentrated for large $m$.

\begin{table}
\begin{center}
   \caption{Values of $ \max_{i} \lbrace \mu_{0}^{(i)} \rbrace $, $ \max_{i,k} \lbrace \mu_{k}^{(i)} \rbrace $ and $\med_{i} \max_{k} \{ \mu_{k}^{(i)} \}$ versus number of points $m$, for the experiment described in Section~\ref{sec:stiefel_experiments}.}\label{tab:1}
   \medskip
   \begin{tabular}{cccccccc}
      \hline\noalign{\smallskip}
          $m$ & 4 & 6 & 8 & 10 & 15 & 20 & 30  \\
      \noalign{\smallskip}\hline\noalign{\smallskip}
         $ \max_{i} \lbrace \mu_{0}^{(i)} \rbrace $   & 0.5577  &  0.7058  &  0.7829  &  0.8296  &  0.8604  &  0.8824  &  0.8980  \\
         $ \max_{i,k} \lbrace \mu_{k}^{(i)} \rbrace $ & 0.8776  &  0.9443  &  0.9671  &  0.9781  &  0.9843  &  0.9881  &  0.9906 \\
         $\med_{i} \max_{k} \{ \mu_{k}^{(i)} \}$      & 0.8774  &  0.9443  &  0.9671  &  0.9781  &  0.9843  &  0.9881  &  0.9906  \\
      \noalign{\smallskip}\hline
            \hline\noalign{\smallskip}
          $m$ & 40 & 50 & 60 & 70 & 80 & 90 & 100  \\
      \noalign{\smallskip}\hline\noalign{\smallskip}
         $ \max_{i} \lbrace \mu_{0}^{(i)} \rbrace $   & 0.9390  &  0.9573  &  0.9728  &  0.9799  &  0.9843  &  0.9870  &  0.9888 \\
         $ \max_{i,k} \lbrace \mu_{k}^{(i)} \rbrace $ & 0.9836  &  0.9799  &  0.9898  &  0.9940  &  0.9959  &  0.9969  &  0.9976 \\
         $\med_{i} \max_{k} \{ \mu_{k}^{(i)} \}$      & 0.9822  &  0.9790  &  0.9898  &  0.9940  &  0.9958  &  0.9968  &  0.9975 \\
      \noalign{\smallskip}\hline
   \end{tabular}
   \end{center}
\end{table}

\section{Conclusions and outlook}\label{sec:conclusions}
We considered the leapfrog algorithm of \protect{\cite{Noakes:1998}} for computing geodesics on Riemannian manifolds and established its convergence as a classical nonlinear Gauss--Seidel method within a numerical linear algebra framework. To keep our discussion concrete, we focused on the case of the Stiefel manifold, which is involved in many fields of application and is challenging because it has no explicit formula available for computing the Riemannian distance. However, our discussion may be generalized to other embedded submanifolds.

Phrasing leapfrog as a nonlinear Gauss--Seidel method could be useful for the numerical analysis community. Indeed, as potential future research directions, we could try using existing ideas, such as acceleration and multigrid, to improve the leapfrog algorithm.
Moreover, it might be possible to analyze leapfrog as a nonlinear block Gauss--Seidel method on the manifold domain directly by using geodesic convexity. This would avoid the need to pick a smooth extension of the distance function.

\section*{Data availability}
The leapfrog algorithm and the single shooting method were implemented in MATLAB, see also~\protect{\cite[chapter~2]{Sutti:2020}}. The software package has been named ``LFMS\_Stiefel'', where ``LFMS'' stands for ``leapfrog multiple shooting'', and can be downloaded from~\url{https://github.com/MarcoSutti/LFMS_Stiefel}. 

\section*{Acknowledgements}
This work was completed during the first author's Ph.D. thesis at the University of Geneva, SNSF fund number 163212.

\appendix

\section{Proof of Lemma~\ref{lemma:expansion_ext_dist_function}}\label{app:proof_expansion_ext_dist_function}
The expansion~\eqref{eq:expansion_ext_dist_function} is simple to obtain once the Riemannian distance is related to the Euclidean one.

\begin{proof}[of Lemma~\ref{lemma:expansion_ext_dist_function}]
Take $X, Y \in \Stnp$ sufficiently close so that we can define the Riemannian logarithm $\xi = \Log_X(Y)$ (see Remark~\ref{rmk:distance}). By definition of the Riemannian distance $d$ for the canonical metric $g$, we have
\[
 d^2(X,Y) = \| \xi \|^{2} = g(\xi,\xi).
\]
Writing a tangent vector as $ \xi = X\Omega + X_{\perp} K \in T_{X}\Stnp $ (see Section~\ref{sec:geodesic_exp_log}) and using~\eqref{eq:formula_canonical_metric}, we can evaluate $g$ as
\[
g(\xi,\xi) = \trace( \xi\tr  (I_{n} - \tfrac{1}{2}\,XX\tr ) \, \xi ) =
\tfrac{1}{2}\,\| \Omega \|^{2}_{\F} + \| K \|^{2}_{\F} = \| \xi \|^{2}_{\F} -
\tfrac{1}{2}\,\| \Omega \|^{2}_{\F}.
\]
Using $\Omega = X\tr \xi$, we also have
\begin{equation}\label{eq:app:deriv_gc}
  d^2(X,Y)  =  \| \xi \|^{2}_{\F} - \tfrac{1}{2}\,\| X\tr \xi \|^{2}_{\F}.
\end{equation} 

Since $\xi$ is the initial velocity vector of the geodesic connecting $X$ to $Y$, it follows that
\begin{equation} \label{eq: xi is Y-X}
   \xi = Y - X + \cO (\| \xi \|^{2}_{\F}).
\end{equation}
This can be seen by expanding the matrix exponential in the expression~\eqref{eq:closed-form-sol-geodesic} of the geodesic:
\begin{align*}
    Y &= \begin{bmatrix}X \ \ X_{\perp} \end{bmatrix} \left( I_{n} + 
        \begin{bmatrix}
            X\tr\xi   &   -\xi\tr X_{\perp} \\
            X_{\perp}\tr\xi        &    O_{n-p}
        \end{bmatrix} +
        O\left(\|\xi\|^{2}_{\F}\right)
        \right)
        \begin{bmatrix}
            I_{p} \\
            O_{(n-p)\times p}
        \end{bmatrix} \\
      &=  X + \begin{bmatrix}X \ \ X_ {\perp} \end{bmatrix} 
              \begin{bmatrix}X \ \ X_ {\perp} \end{bmatrix}\tr \xi
              + \cO (\| \xi \|^{2}_{\F}).
\end{align*}
We obtain~\eqref{eq: xi is Y-X} using the fact that $\begin{bmatrix}X \ \ X_{\perp} \end{bmatrix}$ is an orthogonal matrix. In addition, \cite[Lemma~4.2.1, p.~59]{Belkin:2003} shows that
\begin{equation}\label{eq:belkin}
   \| \xi \|^{2}_{\F} = \| X - Y \|^{2}_{\F} + \cO (\| X - Y \|^{4}_{\F}).
\end{equation}
Then inserting the equations~\eqref{eq: xi is Y-X} and~\eqref{eq:belkin} into~\eqref{eq:app:deriv_gc} leads to
\begin{equation}\label{eq:expansion_canonical_dist}
   d^2(X,Y) = \| X - Y \|^{2}_{\F} - \tfrac{1}{2}\| X\tr (X - Y) \|_{\F}^{2} + \cO (\| X - Y \|^{4}_{\F}).
\end{equation}
Using this in \eqref{eq:ext_dist_function}, one obtains~\eqref{eq:expansion_ext_dist_function}. 

\end{proof}

\section{Proof of Lemma~\ref{lemma:hessian_at_X_Z_on_Stiefel}}\label{app:proof_hessian_at_X_Z_on_Stiefel}
The aim is to compute $ L_{ij} = \nabla_{X_{i}} \nabla_{X_{j}} \widetilde{d}^{2}(X_{i},X_{j}) $ and $D_{ij} = \nabla^2_{X_{i}} \widetilde{d}^{2}(X_{i},X_{j})$, where $X_j \in \Stnp$.  Let us simplify notation and define
\begin{equation*}
\begin{split}
   \widetilde{d}^{2}(X,Y) = & \ \| \PSt{X} - \PSt{Y} \|^{2}_{\F} - \tfrac{1}{2} \| I_{p} - \big(\PSt{X}\big)\tr\PSt{Y}\|^{2}_{\F} + \| {X} - \PSt{X} \|_{\F}^{2} \\
   & + \| {Y} - \PSt{Y} \|_{\F}^{2} + \cO (\| \PSt{X} - \PSt{Y} \|^{4}_{\F}).
\end{split}
\end{equation*}
Clearly, $ L_{ij} = \nabla_{X} \nabla_{Y} \widetilde{d}^{2}(X,Y) $ and $D_{ij} = \nabla^{2}_X  \widetilde{d}^{2}(X,Y)$ with $X = X_i$ and $Y = X_j$.
Recall from Section~\ref{sec:ext_dist_function} that we can specify the projector on the Stiefel manifold as $\PSt(Y) = Y(Y\tr Y)^{-1/2}$, that is, the orthogonal factor of the polar decomposition of $Y$.

\begin{proof}[of Lemma~\ref{lemma:hessian_at_X_Z_on_Stiefel}]
To compute the gradient and the Hessian of $ \widetilde{d}^2(\widetilde{X}, \widetilde{Y}) $, consider the perturbation $\widetilde{X} = X + E $, with $ X \in \Stnp $, $\| E \|_{\F}$ small, and expand the previous expression.

First, for a symmetric matrix $A$, one can easily show by diagonalizing that
\begin{equation*}
   ( I + A )^{-1/2} = I - \tfrac{1}{2} A + \tfrac{3}{8} A^2 + \cO ( \| A \|^{3} ),
   \qquad  \|A\| \to 0,
\end{equation*}
from which we can obtain the expansion for the perturbed projector
\begin{small}
\begin{equation}\label{eq:expansion Pst Xtilde}
\begin{split}
   \PSt\widetilde{X} = & \  \widetilde{X}(\widetilde{X}\tr \widetilde{X})^{-1/2} = \\ 
   & \ X + E -\tfrac{1}{2} XX\tr E - \tfrac{1}{2} X E\tr X - \tfrac{1}{2} X E\tr  E - \tfrac{1}{2} E X\tr  E - \tfrac{1}{2} E E\tr  X  \\
   & + \tfrac{3}{8} X (X\tr  E)^2+ \tfrac{3}{8} X (E\tr  X)^2+ \tfrac{3}{8} XX\tr  E E\tr  X + \tfrac{3}{8} X E\tr  XX\tr  E + \cO ( \| E \|^{3}_{\F} ).
\end{split}
\end{equation}
\end{small}

After substituting the expansion~\eqref{eq:expansion Pst Xtilde} for $\PSt(\widetilde{X})$ in $ \widetilde{d}^{2}(X,Y) $ and isolating first- and second-order terms in $E$, we find the expressions for the gradient and the Hessian. Here, only the final results are reported.

The gradient with respect to $X$ is
\begin{small}
\begin{equation}\label{eq:gradient_wrt_X}
   \nabla_{X} \widetilde{d}^2(X, \widetilde{Y}) = -\big( I_{n} - \tfrac{1}{2} XX\tr \big) \PSt\widetilde{Y} + \tfrac{1}{2} X \big(\PSt\widetilde{Y}\big)\tr X - \big( I_{n} - XX\tr \big) \PSt\widetilde{Y}\big(\PSt\widetilde{Y}\big)\tr X,
\end{equation}
\end{small}
and the gradient with respect to $Y$ is
\begin{small}
\begin{equation*}
   \nabla_{Y} \widetilde{d}^2(\widetilde{X}, Y) = -\big( I_{n} - \tfrac{1}{2} YY\tr \big) \PSt\widetilde{X} + \tfrac{1}{2} Y \big(\PSt\widetilde{X}\big)\tr Y - \big( I_{n} - YY\tr \big) \PSt\widetilde{X}\big(\PSt\widetilde{X}\big)\tr Y.
\end{equation*}
\end{small}

The Hessian matrix with respect to $X$ is
\begin{small}
\begin{equation*}
\begin{split}
\nabla^{2}_{X} & \widetilde{d}^2(X,Y) = \\
      & \Sym\!\Big[ Y\tr\! X \otimes I_{n} + \big( Y\tr \otimes X\big) \, \Pi_{p,n} + I_{p} \otimes YX\tr \Big] \\
      & - \tfrac{3}{4}\Sym\!\Big[\big( Y\tr\! XX\tr \otimes X \big) \, \Pi_{p,n} + \big( X\tr \otimes XY\tr\! X \big) \, \Pi_{p,n} + I_{p} \otimes XX\tr YX\tr + Y\tr\! X \otimes XX\tr \Big] \\
      & + 2\Sym\!\Big[ \big( X\tr YY\tr \otimes X \big) \, \Pi_{p,n} + I_{p} \otimes XX\tr YY\tr - \big(X\tr YY\tr XX\tr \otimes X \big) \, \Pi_{p,n} \Big]  \\
      & + (X\tr \otimes X) \, \Pi_{p,n} + I_{p} \otimes XX\tr - I_{p} \otimes YY\tr + X\tr YY\tr X \otimes I_{n} \\
      & - I_{p} \otimes XX\tr YY\tr XX\tr - X\tr YY\tr X \otimes XX\tr,
\end{split}
\end{equation*}
\end{small}
where $\Sym(A) = (A+A\tr)/2$.
In order to simplify $ \nabla^{2}_{X} \widetilde{d}^2(X, Y) $, we will take $ Y = X + \Delta $ with  $\| \Delta \| \to 0$. After some algebraic manipulations, we obtain\footnote{We stress that $\nabla^{2}_{X} \widetilde{d}^2$ denotes the derivative with respect to the first argument of $\widetilde{d}^2$.}
\begin{equation*}
\begin{split}
\nabla^{2}_{X} & \widetilde{d}^2(X, X + \Delta) = \\
      & \ 2 I_{np} + \tfrac{1}{2} (X\tr \otimes X ) \, \Pi_{p,n} - \tfrac{1}{2}( I_{p} \otimes XX\tr ) + 3\Sym\!\left( X\tr \Delta \otimes I_{n} + ( \Delta\tr \otimes X ) \, \Pi_{p,n} \right)  \\
      & + \Sym(I_{p} \otimes \Delta X\tr) - \tfrac{11}{4}\Sym\!\left( (\Delta\tr XX\tr \otimes X) \, \Pi_{p,n}  + \Delta\tr X \otimes XX\tr \right) \\
      & - \tfrac{3}{4}\Sym\!\left( (X\tr \otimes X \Delta\tr X) \, \Pi_{p,n} + I_{p}\otimes XX\tr \Delta X\tr \right) \\
      & + 2\Sym\!\left( (X\tr \Delta\Delta\tr \otimes X) \, \Pi_{p,n} +  I_{p}\otimes XX\tr \Delta\Delta\tr  - (X\tr \Delta\Delta\tr XX\tr \otimes X) \, \Pi_{p,n} \right) \\
      & - I_{p} \otimes \Delta\Delta\tr + X\tr\Delta\Delta\tr X \otimes I_{n} - I_{p} \otimes XX\tr \Delta\Delta\tr XX\tr - X\tr \Delta\Delta\tr X \otimes XX\tr .
\end{split}
\end{equation*}

Observe that every term on the right-hand side above can be bounded by at most a second power of $ \| \Delta \|_{2} $ since $ \| \Sym(A) \|_{2} \leq \| A \|_{2} $, $\|A\otimes B\|_2 = \|A\|_2\|B\|_2$, and $ X \in \Stnp $. Hence, we obtain after some manipulation that
\begin{equation*} 
    \| \nabla^{2}_{X} \widetilde{d}^2(X, X + \Delta) - \nabla^{2}_{X} \widetilde{d}^2(X, X) \|_2 \leq  14 \| \Delta \|_{2} + 10 \| \Delta \|_{2}^{2}. 
\end{equation*}
Writing the result with $ X = X_i $ and $ X + \Delta = X_j $, we recover the expression~\eqref{eq:Dij_on_Stiefel} for the Hessian $ D_{ij} = \nabla^{2}_{X_{i}} \widetilde{d}^2(X_{i},X_{j}) $.

Next, for the term $ L_{ij} $, to obtain the gradient of \eqref{eq:gradient_wrt_X} with respect to $X$ we can expand $ \PSt \widetilde{X} $ at first order in $ E $ 
\begin{equation*}
   \PSt\widetilde{X} = \widetilde{X}(\widetilde{X}\tr \widetilde{X})^{-1/2} = \ X + E -\tfrac{1}{2} X X\tr E - \tfrac{1}{2} X E\tr  X + \cO ( \| E \|^{2}_{\F} ).
\end{equation*}

After some manipulations, we arrive at the mixed term
\begin{equation*}
\begin{split}
   \nabla_{X} & \nabla_{Y} \widetilde{d}^2(X,Y) = \\
    & - I_{np} + \tfrac{1}{2}(I_{p}\otimes YY\tr) - \tfrac{1}{4}(I_{p}\otimes YY\tr XX\tr)  - \tfrac{1}{4}(X\tr \otimes YY\tr X) \, \Pi_{p,n} \\
   & + \tfrac{1}{2}(I_{p}\otimes XX\tr) + \tfrac{1}{2}(X\tr \otimes X) \, \Pi_{p,n} + \tfrac{1}{2}(Y\tr \otimes Y) \, \Pi_{p,n} - \tfrac{1}{4}(Y\tr XX\tr \otimes Y)\,\Pi_{p,n} \\
   & - \tfrac{1}{4}(Y\tr X \otimes YX\tr) + (Y\tr \otimes YY\tr X) \, \Pi_{p,n} - (Y\tr XX\tr \otimes YY\tr X) \, \Pi_{p,n} \\
   & - Y\tr X \otimes YY\tr XX\tr + Y\tr X \otimes YY\tr - (Y\tr \otimes X) \, \Pi_{p,n} + (Y\tr XX\tr \otimes X) \, \Pi_{p,n} \\
   & + Y\tr X \otimes XX\tr - Y\tr X \otimes I_{n}.
\end{split}
\end{equation*}
We observe that the other mixed term is $ \nabla_{Y}\nabla_{X}\widetilde{d}^2(X,Y) = \big(\nabla_{X}\nabla_{Y}\widetilde{d}^2(Y,X)\big)\tr $. 

As above, in order to bound the spectrum of $ \nabla_{X}\nabla_{Y}\widetilde{d}^2(X,Y) $, we expand it with $ Y = X + \Delta $ with  $\| \Delta \| \to 0$.
After some algebraic manipulations, we obtain
\begin{equation*}
\begin{split}
   \nabla_{X} & \nabla_{Y}\widetilde{d}^2(X,X + \Delta) = \\
   & - 2I_{np} + \tfrac{1}{2}(X\tr\!\! \otimes X) \Pi_{p,n} + \tfrac{3}{2}(I_{p}\otimes XX\tr) - \tfrac{1}{4}(X\tr \!\! \otimes X\Delta\!\tr X) \Pi_{p,n} + \tfrac{1}{2}(\Delta\!\tr \!\! \otimes X) \Pi_{p,n} \\
   &  - \tfrac{1}{4}(\Delta\!\tr XX\tr \otimes X) \Pi_{p,n} + \tfrac{3}{4}(\Delta\!\tr X \otimes XX\tr) - \tfrac{5}{4}(I_{p} \otimes X\Delta\!\tr XX\tr) + \tfrac{3}{2}(I_{p} \otimes X\Delta\!\tr)  \\
   &  - \Delta\!\tr X \otimes I_{n} - \tfrac{5}{4}(\Delta\!\tr XX\tr \otimes \Delta)\Pi_{p,n} - \tfrac{5}{4}(I_{p} \otimes \Delta\Delta\!\tr XX\tr) + \tfrac{3}{2}(I_{p} \otimes \Delta\Delta\!\tr)   \\
   &  + \Delta\!\tr X \otimes X\Delta\!\tr + \tfrac{3}{2}(\Delta\!\tr\!\! \otimes \Delta)\Pi_{p,n} - (\Delta\!\tr XX\tr \otimes X\Delta\!\tr X)\Pi_{p,n} - \tfrac{1}{4}(X\tr\!\! \otimes \Delta\Delta\!\tr X)\Pi_{p,n}  \\
   & + (\Delta\!\tr \otimes X\Delta\!\tr X)\Pi_{p,n} - \tfrac{1}{4}(\Delta\!\tr X \otimes \Delta X\tr) - \Delta\!\tr X \otimes X\Delta\!\tr XX\tr + (\Delta\!\tr\!\! \otimes \Delta\Delta\!\tr X)\Pi_{p,n} \\
   & + \Delta\!\tr X \otimes \Delta \Delta\!\tr - \Delta\!\tr X \otimes \Delta\Delta\!\tr XX\tr -(\Delta\!\tr XX\tr \otimes \Delta\Delta\!\tr X)\Pi_{p,n}.
\end{split}
\end{equation*}
Observe that every term on the right-hand side above can be bounded by at most a third power of $ \| \Delta \|_{2} $. Hence, we obtain that
\begin{equation*}
     \| \nabla_{X}\nabla_{Y} \widetilde{d}^2(X, X + \Delta) - \nabla_{X}\nabla_{Y} \widetilde{d}^2(X, X) \|_{2} \leq  \tfrac{11}{2} \| \Delta \|_{2} + 10 \| \Delta \|_{2}^{2} + 4 \| \Delta \|_{2}^{3}. 
\end{equation*}
Writing the result with $ X = X_i $ and $ Y = X_j $, we recover the expression~\eqref{eq:Lij_on_Stiefel} for the gradient $L_{ij} $.
\end{proof}

\section{Proof of Lemma~\ref{lemma:diagonalization_hessian_on_Stiefel}}\label{app:proof_diag_Hij}
We first start with $i=j$, which corresponds to $\Delta_{ij} = \Lambda_{ij} = 0$ in Lemma~\ref{lemma:hessian_at_X_Z_on_Stiefel}, and prove the following auxiliary lemma. 
\begin{lemma}\label{lemma:diagonalization_Dii_Lii}
Define the orthogonal matrix $~\bar{Q}_i = \big[ I_{p}\otimes X_i \quad I_{p}\otimes X_i^{\perp} \big] \in O(np) $, where $ O(np) $ denotes the orthogonal group of $ np \times np $ matrices.
Then there exists an orthogonal matrix $\widehat{Q}$, only depending on $n$ and $p$, 
such that $ Q_i = \bar{Q}_i \widehat{Q} $ satisfies
\begin{equation}\label{eq:diag_Dij_on_Stiefel_App}
    Q_i\tr D_{ii}Q_i = D = \diag\left\lbrace I_{p(p-1)/2},  \  2 \, I_{np-p(p-1)/2} \right\rbrace,
\end{equation}
and
\begin{equation}\label{eq:diag_Lij_on_Stiefel_App}
    Q_i\tr L_{ii}Q_i = L = \diag\left\lbrace -I_{p(p-1)/2}, \  -2 I_{(n-p)p},  \  O_{p(p+1)/2}  \right\rbrace.
\end{equation}
\end{lemma}
\begin{proof}
By properties of the so-called vec-permutation matrices (see~\cite[Eq.~(5), (6), (23)]{Henderson:1981}), there exists a permutation matrix $\Pi_{p,n} \in \R^{np \times np}$ that satisfies
\[
 \Pi_{p,n} ( X_i\tr  \otimes X_i) \, \Pi_{p,n} =  X_i  \otimes X_i \tr, \quad \Pi_{p,n}^{-1}  = \Pi_{p,n}\tr .
\]
This shows that $( X_i\tr  \otimes X_i) \, \Pi_{p,n} = \Pi_{p,n}\tr ( X_i  \otimes X_i\tr )$ is symmetric. Furthermore,
 \begin{equation*}
\begin{split}
    ( ( X_i\tr  \otimes X_i ) \, \Pi_{p,n} )^2 
     = ( X_i\tr  \otimes X_i ) \, \Pi_{p,n}\,\Pi_{p,n}\tr ( X_i  \otimes X_i\tr ) 
     =  I_{p} \otimes X_iX_i\tr .
\end{split}
\end{equation*}
Denoting the symmetric matrix $S_i = ( X_i\tr  \otimes X_i) \, \Pi_{p,n}$, we can then use Lemma~\ref{lemma:hessian_at_X_Z_on_Stiefel} to write
\begin{equation}\label{eq:Dii_as_poly_Si}
 D_{ii} = 2 I_{np} + \tfrac{1}{2} S_i - \tfrac{1}{2} S_i^2, \qquad 
 L_{ii} = -2 I_{np} + \tfrac{1}{2} S_i + \tfrac{3}{2} S_i^2.
\end{equation}
It thus suffices to diagonalize $S_i$. Using the matrix $ \bar{Q}_i $ defined in the statement of the lemma, direct calculation shows that
\begin{equation*}
\begin{split}
    \bar{Q}_i\tr \, S_i \, \overbar{Q}_i
        & = \diag\left\lbrace  ( X_i\tr \otimes I_{p} ) \, \Pi_{p,n} ( I_{p}\otimes X_i ), \  O_{(n-p)p} \right\rbrace
          = \diag\left\lbrace  \Pi_{p,p},  \  O_{(n-p)p} \right\rbrace \eqqcolon \widehat{\Pi},
\end{split}
\end{equation*}
where we used that $\Pi_{p,n} ( I_{p}\otimes X_i ) \Pi_{p,p} =  X_i \otimes I_{p} $, with $ \Pi_{p,p} \in \R^{p^2 \times p^2}$ another vec-permutation matrix that is also symmetric (see~\cite[Eq.~(6), (15)]{Henderson:1981}). The matrix $\widehat{\Pi}$ above therefore has the spectral decomposition 
\begin{equation}\label{eq:def_hat_Pi}
    \widehat{\Pi} = \widehat{Q} \widehat{\Lambda} \widehat{Q}\tr, \qquad \widehat{\Lambda} = \diag\left\lbrace -I_{p(p-1)/2},  \  O_{(n-p)p},  \  I_{p(p+1)/2}  \right\rbrace
\end{equation}
for some orthogonal matrix $\widehat{Q}$ that indeed does not depend on $X_i$, as claimed. By defining the orthogonal matrix $ Q_i = \bar{Q}_i\widehat{Q} $, we have thus shown that $Q_i\tr  S_i Q_i = \widehat{\Lambda}$ and by~\eqref{eq:Dii_as_poly_Si} also that
\begin{equation*}
    Q_i\tr D_{ii} Q_i = 2 I_{np} + \tfrac{1}{2}\widehat{\Lambda} - \tfrac{1}{2}\widehat{\Lambda}^2, \quad
    Q_i\tr L_{ii} Q_i = -2 I_{np} + \tfrac{1}{2}\widehat{\Lambda} + \tfrac{3}{2}\widehat{\Lambda}^2.
\end{equation*}
It is straightforward to verify that these matrices can be written as the claimed matrices $D$ and $L$.

\end{proof}

Lemma~\ref{lemma:diagonalization_hessian_on_Stiefel} is now proven as a perturbation of the case above.

\begin{proof}[of Lemma~\ref{lemma:diagonalization_hessian_on_Stiefel}] From Lemma~\ref{lemma:hessian_at_X_Z_on_Stiefel}, we know that $L_{ij}  = L_{ii} + \Lambda_{ij} $. Lemma~\ref{lemma:diagonalization_Dii_Lii} therefore gives
\begin{equation*}
 Q_j\tr L_{ij} Q_i = (Q_j - Q_i)\tr L_{ij} Q_i  + Q_i\tr L_{ij} Q_i = (Q_j - Q_i)\tr L_{ij} Q_i + L + Q_i\tr \Lambda_{ij} Q_i.
\end{equation*}
Taking norms and recalling that $\delta_{ij} = \|Q_j - Q_i\|_2$, we obtain
\[
 \| Q_j\tr L_{ij} Q_i - L \|_2 \leq \delta_{ij} (\| L_{ii} \|_2  + \| \Lambda_{ij} \|_2) + \| \Lambda_{ij} \|_2.
\]
Since $ \| L_{ii} \|_2 = \|L\|_2 \leq 2$ by Lemma~\ref{lemma:diagonalization_Dii_Lii}, this shows~\eqref{eq:diag_Lij_on_Stiefel}. The bound~\eqref{eq:diag_Dij_on_Stiefel} is similarly proven. 
\end{proof}

\bibliographystyle{plain} 
\bibliography{leapfrog_arXiv_update_2023.bib}

\end{document}